\documentclass[11pt]{amsart}

\usepackage{amsmath}
\usepackage{fullpage}
\usepackage{xspace}
\usepackage[psamsfonts]{amssymb}
\usepackage[latin1]{inputenc}
\usepackage{graphicx,color}
\usepackage[curve]{xypic} 
\usepackage{hyperref}
\usepackage{graphicx}

\usepackage{amsmath}%
\usepackage{amsthm}%
\usepackage{amscd}
\usepackage{amsfonts}%
\usepackage{amssymb}%
\usepackage{graphicx}

\usepackage{mathrsfs}

\usepackage{tikz}
\usetikzlibrary{matrix,arrows}

\usepackage{tikz-cd}

\newtheorem{theorem}{Theorem}[section]

\newtheorem{conjecture}[theorem]{Conjecture}

\newtheorem{lemma}[theorem]{Lemma}
\newtheorem{notation}{Notation}[section]

\theoremstyle{remark}

\numberwithin{equation}{section}

\newcommand{\gfrak}{\mathfrak{g}}

\newcommand{\Lcal}{\mathscr{L}}
\newcommand{\Mcal}{\mathscr{M}}

\newcommand{\Pro}{\mathbb{P}}

\newcommand{\Z}{\mathbb{Z}}
\newcommand{\C}{\mathbb{C}}

\newcommand{\Q}{\mathbb{Q}}
\newcommand{\R}{\mathbb{R}}

\newcommand{\rk}{\mathrm{rank}\,}

\makeatletter
\@namedef{subjclassname@2020}{%
  \textup{2020} Mathematics Subject Classification}
\makeatother

  \DeclareFontFamily{U}{wncy}{}
    \DeclareFontShape{U}{wncy}{m}{n}{<->wncyr10}{}
    \DeclareSymbolFont{mcy}{U}{wncy}{m}{n}
    \DeclareMathSymbol{\Sha}{\mathord}{mcy}{"58}

\begin{document}
\title[Intersecting the torsion of elliptic curves]{Intersecting the torsion of elliptic curves}

\author{Natalia Garcia-Fritz}
\address{ Departamento de Matem\'aticas,
Pontificia Universidad Cat\'olica de Chile.
Facultad de Matem\'aticas,
4860 Av.\ Vicu\~na Mackenna,
Macul, RM, Chile}
\email[N. Garcia-Fritz]{natalia.garcia@uc.cl}%

\author{Hector Pasten}
\address{ Departamento de Matem\'aticas,
Pontificia Universidad Cat\'olica de Chile.
Facultad de Matem\'aticas,
4860 Av.\ Vicu\~na Mackenna,
Macul, RM, Chile}
\email[H. Pasten]{hector.pasten@uc.cl}%

\thanks{N.G.-F. was supported by ANID Fondecyt Regular grant 1211004 from Chile. H.P. was supported by ANID Fondecyt Regular grant 1230507 from Chile.}
\date{\today}

\begin{abstract} In 2007, Bogomolov and Tschinkel proved that given two complex elliptic curves $E_1$ and $E_2$ along with even degree-$2$ maps $\pi_j\colon E_j\to \mathbb{P}^1$ having different branch loci, the intersection of the image of the torsion points of $E_1$ and $E_2$ under their respective $\pi_j$ is finite. They conjectured (also in works with Fu) that the cardinality of this intersection is uniformly bounded independently of the elliptic curves. As it has been observed in the literature, the recent proof of the Uniform Manin--Mumford conjecture implies a full solution of the Bogomolov--Fu--Tschinkel conjecture. In this work we prove a generalization of the Bogomolov--Fu--Tschinkel conjecture where instead of even degree-$2$ maps one can use any rational functions of bounded degree on the elliptic curves as long as they have different branch loci. Our approach combines Nevanlinna theory with the Uniform Manin--Mumford conjecture. With similar techniques, we also prove a result on lower bounds for ranks of elliptic curves over number fields.
\end{abstract}

\maketitle



\section{Introduction} 

For an elliptic curve $E$ over $\C$ we let $E[\infty]$ be the set of all its torsion points. A morphism $f\colon E\to \Pro^1$ is said to be \emph{even} if for all $P$ in $E$ we have $f(-P)=f(P)$. In 2007, Bogomolov and Tschinkel \cite{BT} used Raynaud's theorem (the Manin--Mumford conjecture) to prove the following:

\begin{theorem}[Bogomolov--Tschinkel] Let $E_1$ and $E_2$ be complex elliptic curves. For each $i=1,2$ let $\pi_j\colon E_j\to \Pro^1$ be a degree-$2$ morphism which is even, and suppose that the branch loci of $\pi_1$ and $\pi_2$ in $\Pro^1$ are different. Then $\pi_1(E_1[\infty])\cap \pi_2(E_2[\infty])$ is finite. 
\end{theorem}

We remark that if $E$ is a complex elliptic curve and $\pi\colon E\to \Pro^1$ is an even degree-$2$ map, its branch locus is precisely $\pi(E[2])\subseteq \Pro^1$. The previous theorem motivated the following conjecture, which seems to have first appeared explicitly in joint works of Bogomolov, Fu, and Tschinkel \cite{BF1,BFT,BT}.

\begin{conjecture}[Bogomolov--Fu--Tschinkel] There is a constant $c$ with the following property. 

For any complex elliptic curves $E_1$ and $E_2$, and even degree-$2$ maps $\pi_j\colon E_j\to \Pro^1$ whose branch loci in $\Pro^1$ are different, we have that $\# \left(\pi_1(E_1[\infty])\cap \pi_2(E_2[\infty])\right)<c$. 
\end{conjecture}

A first breakthrough  was obtained in 2020 by DeMarco, Krieger, and Ye \cite{dMKY} when they proved the conjectured uniform bound in the case $E_1$ and $E_2$ are given in Legendre form $y^2=x(x-1)(x-\lambda)$ and $\pi_j$ are the corresponding projections onto the $x$-coordinate. As noted in \cite{FS}, the Bogomolov--Fu--Tschinkel conjecture is now completely solved thanks to the recent proof of the Uniform Manin--Mumford conjecture \cite{DGH, Kuhne, GGK, Yuan}. See also \cite{Poineau} and \cite{dMM} for alternative proofs of the Bogomolov--Fu--Tschinkel conjecture. In this work we prove the following generalization:

\begin{theorem}[Main Theorem for torsion]\label{ThmA} Let $d$ be a positive integer. There is a constant $c_0(d)$ depending only on $d$ which has the following property: 

For any complex elliptic curves $E_1$ and $E_2$ and non-constant morphisms $g_j\colon E_j\to \Pro^1$ of degree $\deg(g_j)\le d$ whose branch loci in $\Pro^1$ are different, we have that  
$$
\# \left(g_1(E_1[\infty])\cap g_2(E_2[\infty])\right)<c_0(d).
$$
\end{theorem}

In a similar vein, if instead of torsion points we consider the Mordell--Weil group of elliptic curves over number fields, we obtain the following result:

\begin{theorem}[Main Theorem for ranks]\label{ThmB} Let $d$ be a positive integer. There is a constant $\kappa(d)>0$ depending only on $d$ with the following property:

Let $k$ be a number field and let $E_1$ and $E_2$ be elliptic curves over $k$. Let $g_j\colon E_j\to \Pro^1$ be non-constant morphisms defined over $k$ of degree $\deg(g_j)\le d$ with different branch loci. Then
$$
1+\rk E_1(k) + \rk E_2(k)\ge \kappa(d)\cdot \log \max\left\{1,\#\left(g_1(E_1(k))\cap g_2(E_2(k))\right)\right\}.
$$
\end{theorem}

Thus, if two elliptic curves over a number field $k$ have large intersection of the image of their $k$-rational points under rational maps to $\Pro^1$ of fixed degree and different branch loci, then at least one of the two elliptic curves has large rank over $k$. Thus, Theorem \ref{ThmB} is connected to the question of boundedness of ranks of elliptic curves over number fields.

To conclude this introduction let us briefly describe our methods. 

The proof of the Bogomolov--Fu--Tschinkel conjecture applies the Uniform Manin--Mumford conjecture to the curve $X\subseteq E_1\times E_2$ defined by the equation $\pi_1(P_1)=\pi_2(P_2)$ for $(P_1,P_2)\in E_1\times E_2$. For this one checks that $X$ is an irreducible curve of geometric genus at least $2$. 

In the more general setting of Theorem \ref{ThmA} there is no reason for $g_1(P_1)=g_2(P_2)$ to define an irreducible curve in $E_1\times E_2$ and one needs to ensure that all the irreducible components of the resulting algebraic set are curves of geometric genus at least $2$. This is achieved in an indirect way using Nevanlinna theory, for which we review the necessary background in Section \ref{SecNev} ---this approach originates in the authors's proof of Bremner's conjecture \cite{GFP}. In this way we first obtain a purely geometric result in Section \ref{SecGeom} (Theorem \ref{ThmGeom}) from which Theorem \ref{ThmA} is deduced in Section \ref{SecProofs}. Finally, Theorem \ref{ThmB} is also proved in Section \ref{SecProofs} combining our geometric result with the Uniform Mordell--Lang conjecture \cite{DGH, Kuhne, GGK, Yuan} (see Section \ref{SecML}) instead of the Uniform Manin--Mumford conjecture.


\section{Nevanlinna theory} \label{SecNev}

\begin{notation} We will be using Landau's notation $o$. Thus, $o(1)$ represents a function that tends to $0$. In addition, the subscript ``$exc$'' in inequalities and equalities between functions of a variable $r\in \R_{\ge0}$ means that the claimed relation holds for $r$ outside a set of finite measure in $\R_{\ge 0}$.
\end{notation}

 In the first half of the 1920's R.\ Nevanlinna developed a very successful theory to study value distribution of complex meromorphic functions. In this section we recall some basic results of this theory; we refer the reader to \cite{Vojta} for a general reference.

Let $\Mcal$ be the field of (possibly transcendental) complex meromorphic functions on $\C$. Given a non-constant $h\in \Mcal$, a point $\alpha\in \Pro^1(\C)=\C\cup\{\infty\}$, and a real number $r\ge 0$, we define
$$
n_h^{(1)}(\alpha,r)=\#\{ z_0\in \C : |z_0|\le r\mbox{ and }h(z_0)=\alpha\}
$$
where the case $h(z_0)=\infty$ is understood as the condition that $h$ has a pole at $z_0$. The \emph{truncated counting function} $N_h^{(1)}(\alpha,r)$ is then defined as the logarithmic average
$$
N_h^{(1)}(\alpha,r) = \int_{0}^r\left(n_h^{(1)}(\alpha,t)-n_h^{(1)}(\alpha,0)\right)\frac{dt}{t} + n_h^{(1)}(\alpha,0)\log r.
$$

Associated to every $h\in \Mcal$ there is the \emph{Nevanlinna height} (or \emph{characteristic}) function 
$$
T_h\colon \R_{\ge 0}\to \R_{\ge 0}
$$
which has the basic property that it is bounded if $h$ is constant, and otherwise $T_f(r)$ grows to infinity as $r\to \infty$. We recall that the function $T_h(r)$ is only defined up to adding a bounded function. Intuitively, $T_h(r)$ measures the complexity of $h$ restricted to the disk $\{z\in \C: |z|\le r\}$ as $r$ grows.  For our purposes we do not need to recall the precise definition of $T_h(r)$ (which can be found, for instance, in \cite{Vojta}) but instead we simply need its relation to the truncated counting function, which is provided by the \emph{Second Main Theorem} of Nevanlinna theory.

\begin{theorem}[Second Main Theorem]\label{ThmSMT} Let $h\in \Mcal$ be non-constant and let $\alpha_1,\ldots, \alpha_q$ be different points in $\Pro^1(\C)$. Then 
$$
\sum_{j=1}^q N^{(1)}_h(\alpha_j, r)\ge_{exc} (q-2+o(1))T_h(r).
$$
\end{theorem}

Besides the previous general result, we need another relation between the Nevanlinna height and the truncated counting function in a special case.

\begin{lemma}(cf.\ Lemma 3.3 in \cite{GFP})\label{LemmaGFP} Let $E$ be a complex elliptic curve and $g\colon E\to \Pro^1$ a non-constant morphism of degree $d$. Let $\phi\colon \C\to E$ be a non-constant holomorphic map and let $\alpha\in \Pro^1(\C)$. Consider the non-constant meromorphic function $h=g\circ \phi\in \Mcal$. Then
$$
N^{(1)}_h(\alpha,r)=_{exc} \left(\frac{\# g^{-1}(\alpha)}{d}+o(1)\right) T_h(r).
$$
\end{lemma}

We remark that our proof of the previous lemma in \cite{GFP} uses the Second Main Theorem for holomorphic maps to elliptic curves rather than the case of meromorphic functions cited above.


\section{Geometric preliminaries} \label{SecGeom}

Let us fix some notation and assumptions for this section. Let $E_1$ and $E_2$ be complex elliptic curves. For $j=1,2$ let $g_j\colon E_j\to \Pro^1$ be a non-constant morphism of degree $d_j$. Suppose that $g_1$ and $g_2$ have different branch locus in $\Pro^1$. Let $X\subseteq E_1\times E_2$ be the $1$-dimensional algebraic set defined by the equation $g_1(P_1)=g_2(P_2)$ on $(P_1,P_2)\in E_1\times E_2$, that is, $X$ is the pre-image of the diagonal $\Delta\subseteq \Pro^1\times \Pro^1$ via the map $G=(g_1,g_2)\colon E_1\times E_2\to \Pro^1\times \Pro^1$.

If $Z\subseteq E_1\times E_2$ is a $1$-dimensional algebraic set, we define its degree $\deg(Z)$ as the intersection number $Z.(V_1+V_2)$ where $Z$ is seen as a reduced divisor and we define  $V_1=\{0\}\times E_2$ and $V_2=E_1\times \{0\}$ where $0$ is the neutral point of the corresponding elliptic curve. Here we remark that the divisor $V_1+V_2$ on $E_1\times E_2$ is ample.

\begin{lemma}\label{LemmaDeg}  We have $\deg(X)\le (d_1+d_2)d_1d_2$.
\end{lemma}
\begin{proof} Let $\Delta\subseteq \Pro^1\times \Pro^1$ be the diagonal and let $L_1=\{p_1\}\times \Pro^1$ and $L_1=\Pro^1\times \{p_2\}$ for a fixed choice of points $p_1,p_2\in \Pro^1$. Then $\Delta$ is linearly equivalent to $L_1+L_2$.

Note that $X=G^{-1}\Delta \le G^*\Delta$ as divisors, and by the projection formula we find
$$
\begin{aligned}
\deg(X)&\le \deg(G^*\Delta)\\
&=(G^*(L_1+L_2)).(V_1+V_2)\\
&\leq G^*G_*(G^*(L_1+L_2).(V_1+V_2))\\
&=G^*((L_1+L_2).G_*(V_1+V_2))\\
&=d_1d_2(L_1+L_2).(d_2\cdot\{g_1(0)\}\times\mathbb{P}^1+d_1\cdot\mathbb{P}^1\times\{g_2(0)\})\\
&=d_1d_2(d_1+d_2).
\end{aligned}
$$ 
\end{proof}

Before we give the main result of this section, it is convenient to recall a version of the Riemann--Hurwitz formula that will be useful in our argument.

\begin{lemma}[Riemann--Hurwitz formula]\label{LemmaRH} Let $Y_1$ and $Y_2$ be smooth projective (complex) curves of genus $\gfrak_1$ and $\gfrak_2$ respectively, and let $g\colon Y_1\to Y_2$ be a non-constant morphism of degree $d$. Let $\alpha_1,\ldots, \alpha_m\in Y_2$ be all the branch values of $g$. Then
$$
2(\gfrak_1-1)=2d(\gfrak_2-1) + \sum_{j=1}^m (d-\#g^{-1}(\alpha_j)). 
$$
\end{lemma}

With this at hand we have:

\begin{theorem}\label{ThmGeom}  Every irreducible component of $X$ is a curve of geometric genus at least $2$.
\end{theorem}
\begin{proof}
Let $C\subseteq X$ be an irreducible component and for the sake of contradiction suppose that $C$ has geometric genus $0$ or $1$. As $C\subseteq E_1\times E_2$, we see that necessarily $C$ has geometric genus $1$. Since elliptic curves can be uniformized by holomorphic functions from $\C$, we obtain a non-constant holomorphic map $\phi=(\phi_1,\phi_2)\colon \C\to C\subseteq E_1\times E_2$ where at least one of $\phi_j\colon \C\to E_j$ is non-constant. By the definition of $X$ we see that $g_1\circ \phi_1=g_2\circ \phi_2$ and we conclude that both $\phi_j$ are non-constant. 

Let $h=g_1\circ \phi_1=g_2\circ \phi_2\in \Mcal$. Since $g_1$ and $g_2$ have different branch locus, we may assume without loss of generality that there is $\beta\in \Pro^1$ which is a branch value of $g_2$ but not of $g_1$. Let $\alpha_1,\ldots,\alpha_m\in \Pro^1$ be the different branch values of $g_1$. 

By the Second Main Theorem \ref{ThmSMT} with $q=m+1$ we have
$$
N^{(1)}_h(\beta,r)+\sum_{j=1}^m N^{(1)}_h(\alpha_j,r)\ge_{exc} (m-1+o(1))T_h(r).
$$

On the other hand, Lemma \ref{LemmaGFP} gives
$$
N^{(1)}_h(\beta,r) =_{exc} \left(\frac{\# g_2^{-1}(\beta)}{d_2}+o(1)\right)T_h(r)
$$
and for each $1\le j\le m$ we similarly obtain
$$
N^{(1)}_h(\alpha_j,r) =_{exc} \left(\frac{\# g_1^{-1}(\alpha_j)}{d_1}+o(1)\right)T_h(r).
$$
Putting all of this together, we find
$$
\left(\frac{\# g_2^{-1}(\beta)}{d_2}+\sum_{j=1}^m\frac{\# g_1^{-1}(\alpha_j)}{d_1}+o(1)\right)T_h(r)\ge_{exc} (m-1+o(1))T_h(r).
$$
Letting $r\to \infty$, since $h$ is non-constant we deduce
$$
m-1\le \frac{\# g_2^{-1}(\beta)}{d_2}+\sum_{j=1}^m\frac{\# g_1^{-1}(\alpha_j)}{d_1}.
$$
Since $\beta$ is a branch value of $g_2$ we have $\# g_2^{-1}(\beta)<d_2$, hence
\begin{equation}\label{Eqn1}
m-1<1+\frac{1}{d_1}\sum_{j=1}^m \# g_1^{-1}(\alpha_j).
\end{equation}
On the other hand, the Riemann--Hurwitz formula (in the form of Lemma \ref{LemmaRH}) applied to $g_1\colon E_1\to \Pro^1$ gives
$$
0 = -2d_1 + \sum_{j=1}^m(d_1-\# g_1^{-1}(\alpha_j))
$$
from which one gets
$$
 \frac{1}{d_1}\sum_{j=1}^m \# g_1^{-1}(\alpha_j)=m-2.
$$
This contradicts the bound \eqref{Eqn1}.
\end{proof}

\section{Uniform Mordell--Lang and Manin--Mumford} \label{SecML}

The rank of an abelian group $\Gamma$, denoted by $\rk \Gamma$, is defined as the dimension over $\Q$ of the vector space $\Gamma\otimes_\Z \Q$.

After the recent works \cite{DGH, Kuhne, GGK, Yuan}, the Uniform Mordell--Lang conjecture is proved. Here we recall the version obtained in \cite{GGK} in the case of (possibly singular) curves contained in abelian varieties.

\begin{theorem}[Uniform Mordell--Lang for curves]\label{ThmML} Let $n,D\ge 1$ be integers. There is a constant $c(n,D)$ depending only on $n$ and $D$ with the following property.

Let $A$ be an abelian variety over $\C$ of dimension $n$, let $\Lcal$ be an ample line sheaf on $A$, and let $X\subseteq A$ be a $1$-dimensional Zariski closed subset with $\deg_\Lcal(X)\le D$. Let $\Gamma\le A(\C)$ be a subgroup of finite rank and let $r=\rk \Gamma$. If all irreducible components of $X$ have geometric genus at least $2$, then
$$
\# (\Gamma\cap X)\le c(n,D)^{1+r}.
$$
\end{theorem}

As usual, $\deg_\Lcal(X)$ is defined as the intersection number of $\Lcal$ with $X$. Theorem \ref{ThmML} follows from Theorem 1.1 in \cite{GGK}; note that here we do not require that $X$ be irreducible, but this case also follows from \emph{loc.\ cit.} because $\deg_\Lcal(X)$ is additive on $X$ and it is a strictly positive integer as $\Lcal$ is ample.

In the special case where $\Gamma$ is the full torsion subgroup of $A$ one has $r=0$ and the previous result specializes to the Uniform Manin--Mumford conjecture:

\begin{theorem}[Uniform Manin--Mumford for curves]\label{ThmMM} Let $n,D\ge 1$ be integers. There is a constant $c(n,D)$ depending only on $n$ and $D$ with the following property.

Let $A$ be an abelian variety over $\C$ of dimension $n$, let $\Lcal$ be an ample line sheaf on $A$, and let $X\subseteq A$ be a $1$-dimensional Zariski closed subset with $\deg_\Lcal(X)\le D$. Let $A[\infty]$ be the subgroup of all torsion points of $A(\C)$. If all irreducible components of $X$ have geometric genus at least $2$, then
$$
\# (A[\infty]\cap X)\le c(n,D).
$$
\end{theorem}


\section{Torsion and ranks} \label{SecProofs}

In this section we prove Theorems \ref{ThmA} and \ref{ThmB}. For this, let us fix some common notation. Let $k$ be $\C$ in the case of Theorem \ref{ThmA} or a number field in the case of Theorem \ref{ThmB}. Let $E_1$ and $E_2$ be elliptic curves and $g_j\colon E_j\to \Pro^1$ be morphisms of degrees $d_j\le d$ for $j=1,2$, all defined over $k$. We assume that the branch loci of $g_1$ and $g_2$ in $\Pro^1$ are different. 

Let $G=(g_1,g_2)\colon E_1\times E_2\to \Pro^1\times \Pro^1$, let $\Delta\subseteq \Pro^1\times \Pro^1$ be the diagonal, and let $X=G^{-1}\Delta\subseteq E_1\times E_2$. We note that $X$ is the locus of geometric points $(P_1,P_2)$ in $E_1\times E_2$ satisfying $g_1(P_1)=g_2(P_2)$. 

By Lemma \ref{LemmaDeg} we have $\deg(X)\le d_1d_2(d_1+d_2)\le 2d^3$, where $\deg(X)$ is the degree with respect to the ample divisor $V_1+V_2$ as defined in Section \ref{SecGeom}. Furthermore, by Theorem \ref{ThmGeom} the (geometric) irreducible components of $X$ have geometric genus at least $2$.

\begin{proof}[Proof of Theorem \ref{ThmA}] Let $\Gamma=E_1[\infty]\times E_2[\infty]$; this is the group of torsion points of the abelian surface $E_1\times E_2$. By Theorem \ref{ThmMM} we have
$$
\#\left(\Gamma\cap X\right)\le c(2,2d^3)
$$
with $c(n,D)$ as in Theorem \ref{ThmMM}. We note that
$$
g_1(E_1[\infty])\cap g_2(E_2[\infty]) = G(\Gamma)\cap \Delta = G(\Gamma\cap X)
$$
and we obtain the result with $c_0(d)=c(2,2d^3)$.
\end{proof}

\begin{proof}[Proof of Theorem \ref{ThmB}] The proof is very similar. Let $\Gamma=E_1(k)\times E_2(k)$; by the Mordell--Weil theorem this group is finitely generated, and its rank is $r=\rk E_1(k) + \rk E_2(k)$. By Theorem \ref{ThmML} we have
$$
\#\left(\Gamma\cap X\right)\le c(2,2d^3)^{1+r}
$$
with $c(n,D)$ as in Theorem \ref{ThmML}. We note that
$$
g_1(E_1(k))\cap g_2(E_2(k)) = G(\Gamma)\cap \Delta = G(\Gamma\cap X)
$$
and we obtain the result with $\kappa(d)=1/\log c(2,2d^3)$.
\end{proof}



\section{Acknowledgments}

N.G.-F. was supported by ANID Fondecyt Regular grant 1211004 from Chile.

H.P. was supported by ANID Fondecyt Regular grant 1230507 from Chile.

We thank Fabien Pazuki for pointing out some relevant references.


\end{document}